\PassOptionsToPackage{unicode}{hyperref}
\PassOptionsToPackage{naturalnames}{hyperref}
\documentclass[a4paper,11pt]{article}
\usepackage[utf8]{inputenc}

\usepackage{amsfonts, amssymb, amsmath, amsthm}

\usepackage{srcltx}

\usepackage{epsfig}
\usepackage[margin=1.25in]{geometry}
\usepackage{graphicx}
\usepackage{mathpazo}
\usepackage{authblk}

\usepackage{hyperref}

\hypersetup{ 
  unicode=false,          
  pdftoolbar=true,        
  pdfmenubar=true,        
  pdffitwindow=false,     
  pdfstartview={FitH},    
  pdftitle={Classification of partially hyperbolic diffeomorphisms under some rigid conditions},    
  pdfauthor={Pablo D. Carrasco and Enrique Pujals and Federico Rodr\'{i}guez-Hertz},     
  pdfkeywords={Partially Hyperbolic Diffeomorphisms, Classification Conjecture}, 
  pdfnewwindow=true,      
  colorlinks=true,       
  linkcolor=blue,         
  citecolor=green,        
  filecolor=magenta,      
  urlcolor=cyan           
}

\newtheorem{lemma}{Lemma}
\newtheorem{definition}{Definition}
\newtheorem{claim}{Claim}

\newtheorem{theorem}{Theorem}

\newtheorem{proposition}{Proposition}
\newtheorem{remark}{Remark}

\newtheorem*{main-conjecture}{Main conjecture}
\newtheorem{question}{Question}

 \def\RR{{\mathbb R}}

 \def\la{\lambda}

  \def\al{\alpha}
  \def\si{\sigma}

 \newcommand{\diff}{\operatorname{Diff}}

\begin{document}

\title{Classification of partially hyperbolic diffeomorphisms under some rigid conditions}
\author[1]{Pablo D. Carrasco \thanks{pdcarrasco@mat.ufmg.br}}
\author[2]{Enrique Pujals \thanks{epujals@gc.cuny.edu}}
\author[3]{Federico Rodriguez-Hertz \thanks{hertz@math.psu.edu}}
\affil[1]{ICEx-UFMG, Avda. Presidente Antonio Carlos 6627, Belo Horizonte-MG, BR31270-90}
\affil[2]{CUNY Graduate Center, 365 Fifth Avenue, Room 4208, New York, NY10016}
\affil[3]{Penn State, 227 McAllister Building, University Park, State College, PA16802}


\maketitle

\begin{abstract} Consider a three dimensional partially hyperbolic diffeomorphism. It is proved that under some rigid hypothesis on the tangent bundle dynamics, the map is (modulo finite covers and iterates) either an Anosov diffeomorphism, a (generalized) skew-product or the time-one map of an Anosov flow, thus recovering a well known classification conjecture of the second author to this restricted setting.
\end{abstract}

\section{Introduction and Main Results}

Let $M$ be a manifold. One of the central tasks in global analysis is to understand the structure of $\diff^r(M)$, the group of diffeomorphisms of $M$. This is of course a very complicated matter, so to be able to make progress it is necessary to impose some reductions. Typically, as we  do in this article, the reduction consists of studying meaningful subsets in  $\diff^r(M)$, and try to classify or characterize elements on them.   

We will consider partially hyperbolic diffeomorphisms acting on three manifolds. We choose to do so due to their flexibility (linking naturally algebraic, geometric and dynamical aspects), and because of the large amount of activity that this particular research topic has nowadays. Let us spell the precise definition that we adopt here, and refer the reader to \cite{CHHU,HPSurv} for recent surveys.

\begin{definition}
A diffeomorphism of a compact manifold $f:M\rightarrow M$ is partially hyperbolic if there exist a Riemannian metric on $M$ and a decomposition $TM=E^s\oplus E^c\oplus E^u$ into non-trivial continuous bundles satisfying for every $x\in M$ and every unit vector $v^{\sigma}\in E^{\sigma},\ \sigma=s,c,u$,
\begin{itemize}
\item $\|D_xf(v^{s})\|<1,\ \|D_xf(v^{u})\|>1$.
\item $\|D_xf(v^{s})\|<\|D_xf(v^{c})\|<\|D_xf(v^{u})\|$.	
\end{itemize} 
\end{definition}

The set of partially hyperbolic diffeomorphisms on $M$ is a $\mathcal{C}^1$ open set in $\diff^r(M)$. From now on, let $M$ be a three dimensional compact orientable\footnote{By passing to a double cover, this is no loss of generality.} manifold.

We briefly recall some different classes of examples.

\begin{itemize}
 \item Algebraic and geometric constructions. Including, 
   \begin{itemize}
   \item hyperbolic linear automorphisms in the three torus;
   \item skew-products, or more generally circle extensions of Anosov surface maps. By this we mean that there exists a smooth fibration $\pi: M\to \mathbb{T}^2$ with typical fiber $\mathbb{S}^1$, $f$ preserves fibers and the induced map by $f$ on $\mathbb{T}^2$ is Anosov, or
   \item time-one maps of Anosov flows that are either suspensions of hyperbolic surface maps or mixing flows, as the geodesic flow acting on (the unit tangent bundle of) an hyperbolic surface.
  \end{itemize}
 \item Surgery and blow up constructions (which includes the construction of non-algebraic Anosov flows, see \cite{BPP, BGP}). 
\end{itemize}

The motivating question is the following.

\begin{question}
Are the above examples essentially all possible ones, at least modulo isotopy classes? More precisely, is it true that if $f:M\rightarrow M$ is a partially hyperbolic diffeomorphism then it has a finite cover $\widetilde{f}:M\rightarrow M$ (necessarily partially hyperbolic) that is isotopic to one of the previous models? 	
\end{question}

Observe that forgetting the surgery constructions, the first two classes have simple representatives, namely maps whose derivative is constant (with respect to the invariant directions). For example, when $S$ is a compact surface of negative sectional curvature its tangent bundle is an homogeneous space $M=\Gamma\backslash  PSL(2,\RR)$ and the geodesic flow on $M$ is given by right multiplication by $\displaystyle{\begin{pmatrix}
\exp(-\frac{1}{2}t) & 0\\
0 & \exp(\frac{1}{2}t)  \end{pmatrix}}$, so the derivative of each $t$-time map is constant.

In this note we make a contribution to answering the previous question and classify smooth partially hyperbolic maps with constant derivative, or more generally, with constant exponents. A tentative classification of some sort is highly desirable, even in this simplified setting. In that direction,  a classification conjecture by the second author was formulated in 2001 (\cite{Tranph}) and extended  by a modified (weaker) classification conjecture in 2009 due to the third author, J. Rodriguez-Hertz and R. \'{U}res (\cite{CHHU}). Both conjectures turned to be false as proven recently by C. Bonatti, A. Gogolev, K. Parwani and R. Potrie \cite{BPP,BGP}, but as byproduct of the proof, a new zoo of examples was discovered giving another impulse to the research in the topic. Our objective in this paper is then two-fold: on the one hand, prove the above mentioned conjecture in some rigid context and from there, to propose a new possible scheme to classify partially hyperbolic diffeomorphisms on three manifolds, and on the other hand, leave open some questions that may lead to interesting answers.

Given $f:M\rightarrow M$ partially hyperbolic, modulo a finite covering one has that $E^{\sigma}(x)$ is generated by a unit vector field $x\to e_{\sigma}(x)\in \mathbb{R}^3$ for $\sigma=s,c,u$; in other words, there is a finite covering $\hat M$ such that each sub-bundle lifts to an orientable one and therefore the derivative of the lift of $f$ to $\hat M$ acting on $T\hat M$ (that we keep denoting by $f$) can be diagonalized, and so the derivative cocyle $x\to D_xf $ is a cocycle of diagonal elements of $Gl(3,\mathbb{R})$. We denote by $\la_s(x),\la_c(x),\la_u(x)$ the associated eigenvalues. We  say that $f$ has constant exponents if these eigenvalues do not depend on $x$.  Observe that the there are examples of Anosov diffeomorphisms, skew-products over Anosov and Anosov flows (either as suspensions of an Anosov diffeomormorphisms or as Anosov geodesic flows) satisfying that their eigenvalues are constant and having smooth ($\mathcal{C}^{\infty}$) distributions.

\begin{remark}
Notice that our definition of $f$ having constant exponents depends on the chosen metric ($\la(fx)=D_{x}f(e_{\sigma}(x))$). It was pointed out to us by the referee that one can make the definition metric independent by requiring that the (logarithm of the) exponents to be (differentiably) cohomologous to constant. In any case we will work with the metric making all the exponents constant. 
\end{remark}

\begin{theorem}\label{mainth} Let $f$ be a partially hyperbolic $C^{\infty}$ diffeomorphism on a compact orientable 3-manifold $M$ with constant exponents and smooth invariant distributions.

\begin{itemize}
\item If $|\la_c|>1$ then $f$ is $\mathcal{C}^{\infty}$ conjugate to a linear Anosov on $\mathbb{T}^3$.
\item If $|\la_c|=1$ and $f$ is either transitive or real analytic, then there is a finite covering of $M$ such that an iterate of the lift of $f$ is 
\begin{itemize} 
\item $\mathcal{C}^{\infty}$ conjugate to a circle extension of an Anosov linear map, or 
\item a time $t-$map of an Anosov flow of the one following types   

\begin{itemize}
\item the suspension of a two dimensional smooth Anosov map, or
\item $\mathcal{C}^\infty$ leaf conjugate to the geodesic flow of a surface with constant negative curvature, meaning that there exists a smooth diffeomorphism sending the orbits of this Anosov flow to the orbits of the diagonal action on\footnote{Here $\widetilde{SL}(2,\RR)$ denotes the universal covering of $SL(2,\RR)$}  $\Gamma\backslash  \widetilde{SL}(2,\RR)$, for some co-compact lattice $\Gamma$.
\end{itemize}
\end{itemize}
\end{itemize}
\end{theorem}
A sketch of the proof is presented at the beginning of next section.

\begin{remark}
The above theorem implies that under its hypotheses Question 1 has an affirmative answer.
\end{remark}

\begin{question} \label{generalization} Can we get a similar theorem assuming (only) smoothness of the foliations?
\end{question}

Our theorem reveals some inherent rigidity of systems with constant exponents. The reader should compare Theorem 1 with \cite{AVW, Gogolev2017, Gogolev2019, SY}, where rigidity results are obtained for some perturbations of the listed maps (time-one maps of geodesic flows, Anosov diffeomorphisms and skew-products).

\smallskip

About the tentative classification without any extra assumption beyond partial hyperbolicity, it has been proved recently (see also \cite{Po}):
 
\begin{itemize}
 \item[--] Partially hyperbolic diffeomorphims in Seifert and Hyperbolic manifolds are conjugate to a discretized topological Anosov flow (see \cite{BFFP}); also it was announced by R. \'{U}res when $M= T^1S$ ($S$ is a surface) assuming that $f$ is isotopic to the geodesic flow through a path of partially hyperbolic diffeomorphisms.

\item[--]  If $M$ is a manifold with (virtually) solvable fundamental group an $f-$invariant center foliation, then (up to finite lift and iterate) it is leaf conjugate to an algebraic example (see \cite{HP1, HP2}).

\item[--] In \cite{BPP,BGP}, using surgery it was constructed a large family of new partially hyperbolic examples that are not isotopic to any one in the thesis of theorem \ref{mainth}. See also the blow-up constructions in \cite{Gogolev2016}.
\end{itemize}

\begin{question} How does Theorem \ref{mainth} relate to the above mentioned recent results?\footnote{After completing this manuscript we received a preprint from C. Bonatti and J. Zhang where they obtain a $\mathcal{C}^0$ rigidity result assuming neutral center \cite{BZ}.}
 
\end{question}

Related to a general classification, would it be possible that the rigid ones are kind of ``building blocks'' from where any 3-dimensional partially hyperbolic one ``is built''?

\begin{question} Given an compact orientable three manifold $M$ and $f:M\rightarrow M$  partially hyperbolic, is it true that $M$ ``can be cut'' into finitely many (manifold with boundary) pieces $M_1,\ldots, M_k$ such that  $M_i$ is an open submanifold in a compact three manifold $\hat M_i$ carrying $f_i\in\mathcal{PH}(\hat M_i)$ with constant exponents, and so that $f|M_i$ is isotopic (relative to  $M_i$) to $f_i|\hat M_i$?
\end{question}

\section{Proof of the Main Result.}

To avoid repetition, from now on we assume that all the sub-bundles are orientable and that we are working in the corresponding lift (as it was mentioned before its statement, the main theorem holds up to a finite covering).

Since the distributions $E^{\sigma}$ are differentiable, they are uniquely integrable to one dimensional foliations $\mathcal{F}^{\sigma}$ of $\mathcal{C}^{\infty}$ leaves.  Consider the orthonormal invariant (ordered) base $\mathcal{B}(x)=\{e^u(x),e^s(x),e^c(x)\}$ referred in the introduction, and denote by $A(x)$ the associated matrix to $D_xf$ in the bases $\mathcal{B}(x),\mathcal{B}(fx)$. By hypotheses $A(x)=A\in Gl(3,\RR)$ is diagonal, hence it is partially hyperbolic with  determinant $\pm 1$, and thus is an hyperbolic matrix or it has one eigenvalue of modulus one. In the former case $f$ is Anosov, while in the later $f$ acts as an isometry on its center.

Let $\phi_t^s,\phi^c_t,\phi_t^u$ be the flows that integrate the bundles $E^s,E^c,E^u$ parameterized by arc-length (in short, we refer them as $\phi^\sigma_t$ with $\sigma=s,c,u$). By hypotheses, these are $\mathcal{C}^{\infty}$ flows.

\begin{question}
Is the smoothness hypothesis on the bundles necessary in presence of constant exponents?
\end{question}

The poof of the theorem goes at follows. If $|\la_c|\neq 1$, $f$ is Anosov and it is  constructed global $\mathcal{C}^{\infty}$ coordinates to show that $f$ is $\mathcal{C}^{\infty}$ conjugate to a linear Anosov map; if $|\la_c|=1$ by the commutation of $\phi^c_t$ and $f$ (see equations \ref{comm2} and  \ref{comm3}) it follows that  $D\phi^c_t$ is constant in the corresponding $f$-invariant base (see lemma \ref{constant}), and therefore it is either the identity or partially hyperbolic. In the first case all the center leaves are compact and then $f$ is an extension of a two dimensional Anosov (see Proposition \ref{identity}), while in the second $\phi^c_t$ is an Anosov flow and there is $T$ such that $f=\phi_T^c$ (see lemma \ref{differential}). Moreover, by \cite{ghyss} it holds that $\phi^c_t$ is (modulo coverings and reparametrizations) the geodesic flow of surface with constant negative curvature or the suspension of a linear Anosov with constant time. We point out that in \cite {Tranph} it is concluded that under transitivity, a three dimensional partially hyperbolic is either a skew-product or an Anosov flow, assuming the existence of certain type of periodic trajectories for $\phi^c_t$ and some properties on the dynamics of the homoclinic points associated to these periodic orbits. 

For perturbations of the linear Anosov map, the same result may also be obtained by using the first theorem in \cite{SY} once it is shown that the exponents of the Anosov and its linear part are the same, which can be deduced from quasi-isometry of the foliations. A different approach to prove smooth conjugacy to a linear Anosov model was developed in \cite{V}, that uses smoothness of the center foliation plus extra requirements about the stable/unstable holonomies; to apply that approach one may have to establish that the hypotheses of our main theorem imply the requirements of \cite{V}, which doesn't seem to be direct. Other result related to the case that $|\lambda_c|=1$ is the one proved in \cite{AVW}: any partially hyperbolic diffeomormorphisms (that preserves a Liouville probability measure) close to the time-one map of a geodesic flow of a  negative curved surface with a smooth center foliation is the time-one map of a flow (close to the geodesic flow).

Given $x$, since $f$ preserves the three foliations, it holds that 
\begin{eqnarray*}\label{comm1}
 f\circ \phi^\si_t(x)= \phi^\si_{\lambda_\si.t}\circ f(x),
\end{eqnarray*}
where $\lambda_\si$ is the  eigenvalue of $Df$ along $E^\si$. The same equations leads to 
 \begin{eqnarray}\label{comm2}
  f^n\circ \phi^\si_t(x)= \phi_{\lambda_\si^n.t}\circ f^n(x).
\end{eqnarray}
In particular, it holds that

\begin{equation}\label{comm3}
   D_{\phi^\si_t(x)}f^n\circ D_x\phi^\si_t= D_{f^n(x)}\phi^\si_{\lambda_\si^n.t}\circ D_xf^n.
   \end{equation}
Differentiating \eqref{comm3} with respect to $t$ we get the following equation:

\begin{eqnarray*}\label{comm40}
\partial_tD_{\phi^\si_t(x)}f^n\circ D_x\phi^\si_t+D_{\phi^\si_t(x)}f^n\circ\partial_tD_x\phi^\si_t= \lambda_\si^n. \partial_tD_{f^n(x)}\phi^\si_{\lambda_\si^n.t}\circ D_xf^n, 
\end{eqnarray*}
hence if we denote by $B^\si(x,t)$ the associated matrix to $D_x\phi^\si_t$ in the bases $\mathcal{B}(x)$, $\mathcal{B}(\phi^\si_t(x))$ we obtain, using that the representation of $D_xf^n$ ($=A^n$) is independent of time,  
\[
A^n\cdot\partial_t B^\si(x,t)=\lambda_\si^n\partial_t B^\si(f^n(x),\lambda_\si^n.t)\cdot A^n,
\]
therefore by fixing $t_0$, it holds 
\begin{eqnarray}\label{comm4}
A^n\cdot\partial_tB^\si(x,t)|_{t=\frac{t_0}{\la_c^n}}\cdot A^{-n}=  \lambda_\si^n\partial_tB^\si(f^n(x),t)|_{t=t_0}.
\end{eqnarray}	
Since $D\phi^\si_t(E^\si)=E^\si$, the two non-diagonal terms of the corresponding column of $B^\si(x,t)$ are zero, thus the same is true for $\partial_tB^\si(x,t)$.

We divide the argument into cases depending on whether $\lambda_c>1$ or $\lambda_c=1$.

\subsection{\texorpdfstring{$\lambda_c>1$}: Anosov case}

First we consider the case $\lambda_c>1$.  Clearly, $f$ is Anosov and therefore it is conjugate (in the $\mathcal{C}^0$ category) with its linear part $L:\mathbb{T}^3\rightarrow\mathbb{T}^3$; i.e. there exists  $L \in Sl(3,\mathbb{Z})$ with invariant bundles $E^s_L, E^{c}_L, E^{u}_L$ and exponents $\gamma_s<1<\gamma_c<\gamma_{u}$ conjugate to $f$.  The goal is to show that the conjugacy with the linear part is actually smooth. To do that, it is revisited   the classical result of Franks \cite{franksthe} that use the foliations to build the conjugacy along the following steps:

\begin{itemize}
 \item[--] it is  considered the lift of $f$ to $\RR^3$, which after conjugating by a translation can be assumed that $f(0)=0$ and the lifts of the foliations that integrates the invariant sub-bundles; those foliations,   provide a $\mathcal{C}^{\infty}$ system of coordinates; i.e., any point $x$ can be written as $(x^s, x^c, x^u)$ with $x^\si\in {\mathcal F}^\sigma(0)$ (the invariant leaves at the point $(0,0,0))$;
 
  \item[--] it is shown that $f$ can be ``linearized'', in the sense that $f$ can be written as $f(x^s, x^c, x^u)= (f^s(x^s), f^c(x^c), f^u(x^u))$ where $f^\sigma: {\mathcal F}^\sigma(0) \to {\mathcal F}^\sigma(0)$ is a smooth diffeomorphism;

   \item[--] each one dimensional diffeomorphism $f^\si$ is $\mathcal{C}^{\infty}$ conjugate to $L|{E^\si_L}$  by a $\mathcal{C}^{\infty}$ diffeomorphism $h^\sigma: {\mathcal F}^\si(0)\to E^\si_L;$
   
   \item[--] the $\mathcal{C}^{\infty}$ diffeomorphism $h=(h^s, h^c, h^u)$  is a conjugacy between $f$ and $L.$

\end{itemize}

For the first part, we first remark that as consequence of the classical stable manifold theorem, the bundle $E^c\oplus E^{u}$ is also integrable to an $f$-invariant foliation ${\mathcal F}^{cu}$, the center unstable foliation. In the lift to $\RR^3$, for any point $x$ there are unique points $x^s\in {\mathcal F}^s(0)$ and $x^{cu}\in {\mathcal F}^{cu}(0)$ such that $x\in {\mathcal F}^{cu}(x^s)\cap {\mathcal F}^{s}(x^{cu})$ and any point in $x^{cu}\in {\mathcal F}^{cu}(0)$ there are  unique points $x^c\in {\mathcal F}^c(0)$ and $x^{u}\in {\mathcal F}^{u}(0)$ such that $x^{cu}\in {\mathcal F}^{u}(x^c)\cap {\mathcal F}^{c}(x^{u})$. On that way, it is obtained a $\mathcal{C}^{\infty}$ system of coordinates and any point can be written as $(x^s, x^c, x^u).$

For the second item, first observe that using the linear coordinates it follows that $f$ is expressed as $f(x^s, x^c, x^u)= (f^s(x^s, x^c, x^u), f^c(x^s, x^c, x^u), f^u(x^s, x^c, x^u));$ so, the goal is to show that $f^\si$ only depends on the $x^\si-$coordinate. For that it is enough to show that all the holonomies preserve the invariant sub-bundles and this is done showing the derivative of  $\phi_t^\si$ are the identity.
We'll consider $\sigma=c$, as the other cases are completely analogous. Writing  $\partial_tB^c(x,t)\Big|_{t=\frac{t_0}{\la_c^n}}=(a_{ij}), \partial_tB^c(f^n(x),t)|_{t=t_0}=(a_{ij}')$
and using  \eqref{comm4} one gets
\[
\begin{pmatrix}
a_{11} & \left(\frac{\la_u}{\la_s} \right)^n\cdot a_{12} & 0\\
\left(\frac{\la_s}{\la_u} \right)^n\cdot a_{21} & a_{22} & 0\\
\left(\frac{\la_c}{\la_u} \right)^n \cdot a_{31} & \left(\frac{\la_c}{\la_s} \right)^n\cdot a_{32} & a_{33}
\end{pmatrix}=\la_c^n\begin{pmatrix}
a_{11}' & a_{12}' & 0\\
a_{21}' & a_{22}' & 0\\
a_{31}' & a_{32}' & a_{33}'
\end{pmatrix}.
\]
Observe that the coefficients $a_{ij}, a_{ij}'$ are bounded with $n$, while  $\partial_tD_x\phi^c_t|_{t=\frac{t_0}{\la_c^n}}\rightrightarrows \partial_tD_{x}\phi^c_{t}|_{t=0}$ uniformly as $n\to\infty$; using this and the relation $\la_s<1<\la_c<\la_u$ one deduces that $\partial_tD_{x}\phi^c_{t}|_{t=0}$ is the zero matrix. Finally, it is well know that $f$ has dense orbits, hence by taking one of these we deduce that $\partial_tB^c(x,t)|_{t=t_0}=0$ for every $x\in M,t_0\in\mathbb{R}$. This implies that $B^c(x,t)$ is the identity matrix for every $t,x$, and in particular $D\phi^c_{t}(E^{\sigma})=E^{\sigma}, \sigma=u,s,c$. 

The argument above works similarly for the flows $\phi^u,\phi^s$, interchanging $\la_c$ by $\la_u,\la_s$ (which are different from one) thus establishing the second item.

To prove the third item, it is enough to show that the eigenvalues of $L$ are the same of $f$:

\begin{claim}
	It holds $\gamma_{u}=\la_u,\gamma_c=\la_c, \gamma_s=\la_s$.
\end{claim}

\begin{proof}
Since the topological entropy of $f$ and $L$ are the same we obtain $\gamma_s=\la_s, \gamma_{u}+\gamma_c=\la_u+\la_c$. Using that the conjugacy between $f$ and $L$ sends 
$\mathcal{F}^c$ to $\{E^c_L+x\}_{x\in\mathbb{T}^3}$, one deduces $\gamma_c=\la_c$ which finishes the claim. 
\end{proof}

Now, one can  define $h^{\sigma}: {\mathcal F}^{\sigma}(0)\rightarrow E^{\sigma}_L$ by
\[
h^{\sigma}(x)=\text{oriented arc length in }W^{\sigma}(0) \text{ of the shortest interval between }0 \text{ and }x. 
\]

Each $h^{\sigma}$ is a $\mathcal{C}^{\infty}$ diffeomorphism, and since all hololonomies corresponding to invariant foliations of $f$ are the identity, they assemble to a $\mathcal{C}^{\infty}$ diffeomorphism  $h:\mathbb{R}^3\rightarrow \mathbb{R}^3$. By the previous claim $h$ conjugates the action of $f$ with $L$ concluding that  {\em $f$ is $\mathcal{C}^{\infty}$ conjugate to its linear part.}

\begin{remark}
If one  assumes that $f$ has constant derivative (i.e.\@ the invariant bundles are constant), then the above argument is simplified concluding that  $f=L$.
\end{remark}

\subsection{\texorpdfstring{$\lambda_c=1$}: generalized skew-products}

Now we consider the case $\lambda_c=1$. As in previous case, it is shown that $D\phi^c_t$ preserves the sub-bundles, however, since now the center eigenvalue is one, it is needed a different proof.

\begin{lemma}
It holds $D\phi^c_t(E^{\sigma})=E^{\sigma}$ for $\sigma=u,s,c$.
\end{lemma}

\begin{proof}
We consider the case $\sigma=u$ only, as the argument for $\sigma=s$ is completely analogous (while $\sigma=c$ is direct consequence of $E^c$ being tangent to flow lines). Fix $x\in M, y=\phi^c_t(x)$ and take $v=D_y\phi^c_t(e^u(x))$. By integrability of $E^u\oplus E^c$ we can write $v=ae^u(y)+be^c(y)$. Using \eqref{comm3} with $n>0$ and since distances along centers are preserved, we get
\[
D_{f^{-n}(x)}\phi^c_{t}\Big(\frac{1}{\la_u^n}e^u(f^{-n}(x))\Big)=\frac{a}{\la_u^n}e^u(f^{-n}(y))+be^c(f^{-n}(y)).
\]
This gives a contradiction for $n$ large, unless $b=0$. 
\end{proof}

As in the previous part denote by $B^c(x,t)$ the associated matrix to $D_x\phi^c_t$ in the corresponding invariant bases. By \eqref{comm3} $A^n\cdot B^c(x,t)\cdot A^{-n}=B^c(f^n(x),t)$, and since all matrices are diagonal this implies $B^c(x,t)=B^c(f^n(x),t)$ $\forall n.$

\begin{lemma}\label{constant} If $f$ is transitive or real analytic then $B^c(x,t)$ is constant in $x$.
\end{lemma} 

\begin{proof} This follows directly by the previous equality (invariance of $B^c(x,t)$ in the orbit of $x$), either by taking a dense orbit (in the transitive case) or a recurrent trajectory (which exists by Birkhoff's recurrence theorem) in the real analytic case, by the zeros theorem for analytic functions.
\end{proof}

We deduce that for $t$ fixed the map $\phi^c_t$ is conservative with constant exponents, hence $B^c(x,t)$ is either
\begin{itemize}
\item[--] the identity, or
\item[--] partially hyperbolic (a center eigenvalue equal to $1$, one larger and other smaller). In this case $\phi^c_t$ is an Anosov flow.
\end{itemize}

The case when $B^c(x,t)=Id$ is the simpler one.

\begin{proposition}\label{identity}
If  $B^c(x,t)=Id$  then it holds:
\begin{itemize}
 \item[--] all the center leaves are closed, and
 \item[--]  $f$ is $\mathcal{C}^{\infty}$ conjugate to a circle extension of a linear Anosov map in $\mathbb{T}^2$.
\end{itemize}
\end{proposition}

We will prove this through a series of lemmas.

\begin{lemma}\label{identititylem1}
If  $B^c(x,t)=Id$  then there is a closed center leaf (i.e.\@ a circle tangent to $E^c$).	
\end{lemma}
	
\begin{proof}
Taking a recurrent point one can find $p$ such that ${\cal F}^c_p$ is invariant by $f^k$ for some $k$. We claim that ${\cal F}^c_p$ is closed. Assuming otherwise, ${\cal F}^c_p$ is homeomorphic to the real line and so $f^k:{\cal F}^c_p\to {\cal F}^c_p$ is either the identity or a translation.  Observe that for a partially hyperbolic diffeomorphism, two periodic points of the same period that are sufficiently close have to belong to the same local center manifold. But, if ${\cal F}^c_p$ is not closed and $f^k|{{\cal F}^c_p}$ is the identity, there are periodic points of $f$ with  the same period, arbitrary close one to each other that does not share the same local center leaf. In case that $f^k|{{\cal F}^c_p}$ is a translation, i.e. $f^k(x)=x+\alpha$ along the center leaf; one can take a point $z$ and $t$ arbitrary large such that $z, \phi_{t}(z), \phi_{2.t}(z)$ are  close to each other and arcs $I_0, I_1, I_2$ with length $4.\alpha$ inside ${\cal F}^c_p$ and containing in the middle the points  $z, \phi_{t}^c(z), \phi_{2.t}^c(z)$ respectively; since $t$ is large, the three arcs are disjoints. Let $n$ be the smallest positive integer  such that $f^{k.n}(z)\in I_1$, which exists because $f^k$ restricted to the center is a translation by $\alpha$ and the arcs has length $4.\alpha$. From the commutative property, also holds that $f^{2.k.n}(z)\in I_2$; in particular, $f^{k.n}(I_0)\cap I_1\neq\emptyset$ and  $f^{k.n}(I_1)\cap I_2\neq\emptyset$. Since $f^k$ is partially hyperbolic, the unstable distance of $I_2$ to $I_1$ is $\lambda^u$ times the distance from $I_1$ to $I_0$. On the other hand, since $\phi^c_t(I_0)=I_1$ and $\phi^c_t(I_1)=I_2$ and $D\phi^c_t$ is the identity, it holds that the unstable distance of $I_2$ to $I_1$ is equal to the distance from $I_1$ to $I_0$. A contradiction.
\end{proof}

\begin{lemma}\label{identititylem2}
	If  $B(x,t)=Id$ then all center leaves are closed.	
\end{lemma}

\begin{proof}	
By the previous Lemma there exists a closed center leaf, thus there is a periodic point $x$ of $\phi^c_t.$ Let us consider two local transversal sections $\Sigma'\subset \Sigma$ to the flow containing $x$ and let $R$ be the first return map from $\Sigma'$ to $\Sigma.$ The transversal section can be taken in such a way that $T_y\Sigma=N_y$ where $N_y$ is the orthogonal plane to the flow direction at $y$. In that case, $D_yR$, the derivative of $R$ at a point $y\in \Sigma$, coincides with $\hat \phi_{r(y)}(y)$, the Linear Poincar\'{e} flow at $y$ with $r(y)$ being the return time of $y$ to $\Sigma$ by the flow $\phi^c_t$. Therefore, for any $y\in \Sigma'$, the derivative of the return map is the identity and since $R$ has a fixed point, then $R$ is the identity in $\Sigma'$. In particular, this implies that any center leaf intersecting $\Sigma'$ is a closed leaf with trivial holonomy. This way we prove that the set of points having a closed center leaf is an open set. Since the center eigenvalue of $f$ is one, we deduce that for a point $p$ having a compact center leaf all other leaves inside $W^{cs}(p),W^{cu}(p)$ are circles with uniformly bounded length, and this implies that for a given closed center leaf there exists an open set of bounded by below diameter where all other center leaves are closed. Since the recurrent points of $f$ are dense (because $f$ is conservative), we deduce that every center leaf is closed.
\end{proof}

\begin{proof}[Proof of Proposition \ref{identity}]

By the Lemma above  ${\cal F}^c$ is a $C^\infty$ foliation by compact leaves without holonomy and so $M/{\cal F}^c$ is a smooth compact surface and $M\to M/{\cal F}^c$ is a smooth fibration.  By standard arguments it follows that $M$ is a Nilmanifold (see for example Theorem 3 in \cite{maxmeas}).
The map $f$ induces an hyperbolic diffeomorphism $\hat f:M/{\cal F}^c\to M/{\cal F}^c$ that has constant exponents in the base obtained by projecting $\{\mathcal{B}(x)\}$. By the same arguments used in the case $\la_c>1$ we deduce that $M/{\cal F}^c$ is the two dimensional torus and $\hat f$ is $\mathcal{C}^{\infty}$ conjugate to a linear Anosov $L$. By extending the aforementioned conjugacy to $M$ as the identity in the fibers, we conclude that $f$ is $\mathcal{C}^{\infty}$ conjugate to an extension of $L$.

\end{proof}

\begin{question}
In the skew-product case, $M=\mathbb{T}^3$ and $f$ is conjugate to a map of the form $L\rtimes g_x$,  $L(x,\theta)=(L\cdot x,\theta+\alpha(x))$. It was asked by the referee which type of properties can be deduced from $\alpha$ if we assume further that the invariant bundles are smooth, so we leave the problem for the interested reader.  
\end{question}

\subsection{\texorpdfstring{$\lambda_c=1$}: Anosov flow case}

It remains for us to analyze the case where $D_x\phi^c_t$ is partially hyperbolic.

\begin{lemma}\label{differential}
If $D_x\phi^c_t$ is partially hyperbolic then $\phi^c_t$ is either the suspension of a $\mathcal{C}^{\infty}$ Anosov map in $\mathbb{T}^2$ or, modulo finite covering and $\mathcal{C}^{\infty}$ conjugacy, the geodesic flow acting on a surface of constant negative sectional curvature.
\end{lemma}

\begin{proof}
We already saw that $\phi^c_t$ is an Anosov flow with $\mathcal{C}^{\infty}$ stable and unstable distributions. Either $\phi^c_t$ is a suspension (necessarily of $\mathcal{C}^{\infty}$ Anosov surface map), or by \cite{ghyss} there exists a smooth diffeomorphism sending the orbits of $\phi^c_t$ to the orbits of the diagonal flow on a homogeneous space $\Gamma\backslash  \widetilde{SL}(2,\RR)$. 
\end{proof}

\begin{proposition}\label{flow}
If  $D_x\phi^c_t$ is partially hyperbolic then there exists an iterate $f^k$ that is the time $T$-map of an Anosov flow.
\end{proposition}

We first note the following.

\begin{lemma}
If  $D_x\phi^c_t$ is partially hyperbolic then there is an iterate $f^k$ and a closed center leaf $O(p)$ such that modulo a $\mathcal{C}^{\infty}$ reparametrization of $\phi_t^c$, it holds:
\begin{itemize}
\item $O(p)$ has length one.
\item If $W^s_{loc}(O(p), \phi^c_t),  W^u_{loc}(O(p), \phi^c_t)$ are the local stable and unstable manifolds of $O(p)$ with respect to $\phi^c_t$ then 
   \begin{itemize}
   \item $f^k(W^s_{loc}(O(p), \phi^c_t))\subset W^s_{loc}(O(p), \phi^c_t)$, and
   \item $f^{-k}(W^u_{loc}(O(p), \phi^c_t))\subset W^u_{loc}(O(p), \phi^c_t)$.
   \end{itemize}
\end{itemize}	  
\end{lemma}

\begin{proof} As noted above, $\phi^c_t$ is conservative. Since $\phi^c_t$ is a hyperbolic flow, there exists at most finitely many shortest closed orbits. Let $O(p)$ be one of these shortest closed curves.  Since  $f(O(p))$ is a compact leaf of the same length, $O(p)$ is a periodic curve of $f$. It follows that there is a positive integer $k$ such that $f^k(O(p))=O(p)$. We reparametrize the flow so that $O(p)$ has length $1$, i.e.\@ $\phi_1^c(z)=z\ \forall z\in O(p)$. 
Since the only $f^k$-invariant sets near $O(p)$ are $W^{cs}_{loc}(p,f^k)$ and $W^{cu}_{loc}(p,f^k)$ (the center stable and center unstable manifolds of $p$), we have that $f^k$ permutes the set $\{W^s_{loc}(O(p), \phi^c_t),W^u_{loc}(O(p), \phi^c_t)\}$, hence by changing $t$ by $-t$ if necessary we can assume that $f^k(W^s_{loc}(O(p), \phi^c_t))\subset W^s_{loc}(O(p), \phi^c_t)$ and $f^{-k}(W^u_{loc}(O(p), \phi^c_t))\subset W^u_{loc}(O(p), \phi^c_t)$.
\end{proof}

We continue working with $O(p)$ given in the lemma and assume that $f(O(p))=O(p)$ (so, the actual result is about $f^k$ and not $f$). Note that both $W^s(O(p),\phi^c_t)$ and $W^u(O(p,\phi^c_t))$ are cylinders over $O(p)$. We introduce (linearizing) coordinates $(\theta,x)$ in $W^s_{loc}(O(p), \phi^c_t)$ and $(\theta,y)$ in $W^u_{loc}(O(p), \phi^c_t)$ with $\theta\in \mathbb{R}/\mathbb{Z}$ and $x,y\in [-\lambda_u,\lambda_u]$. Consider the curves $\gamma_s=\{(\theta,x):x=1\},\gamma_u=\{(\theta,x):y=\lambda_u\}$ and note that they are transverse to $\phi^c_t$. Finally consider the fundamental domains $D^s\subset  W^s_{loc}(O(p), \phi^c_t), D^u\subset W^u_{loc}(O(p), \phi^c_t)$  delimited by $\gamma_s,f(\gamma_s)$ and $\gamma_u, f(\gamma_u)$ respectively. See the picture below.

\begin{figure}[ht]
	\centering
	\includegraphics[width=0.7\linewidth]{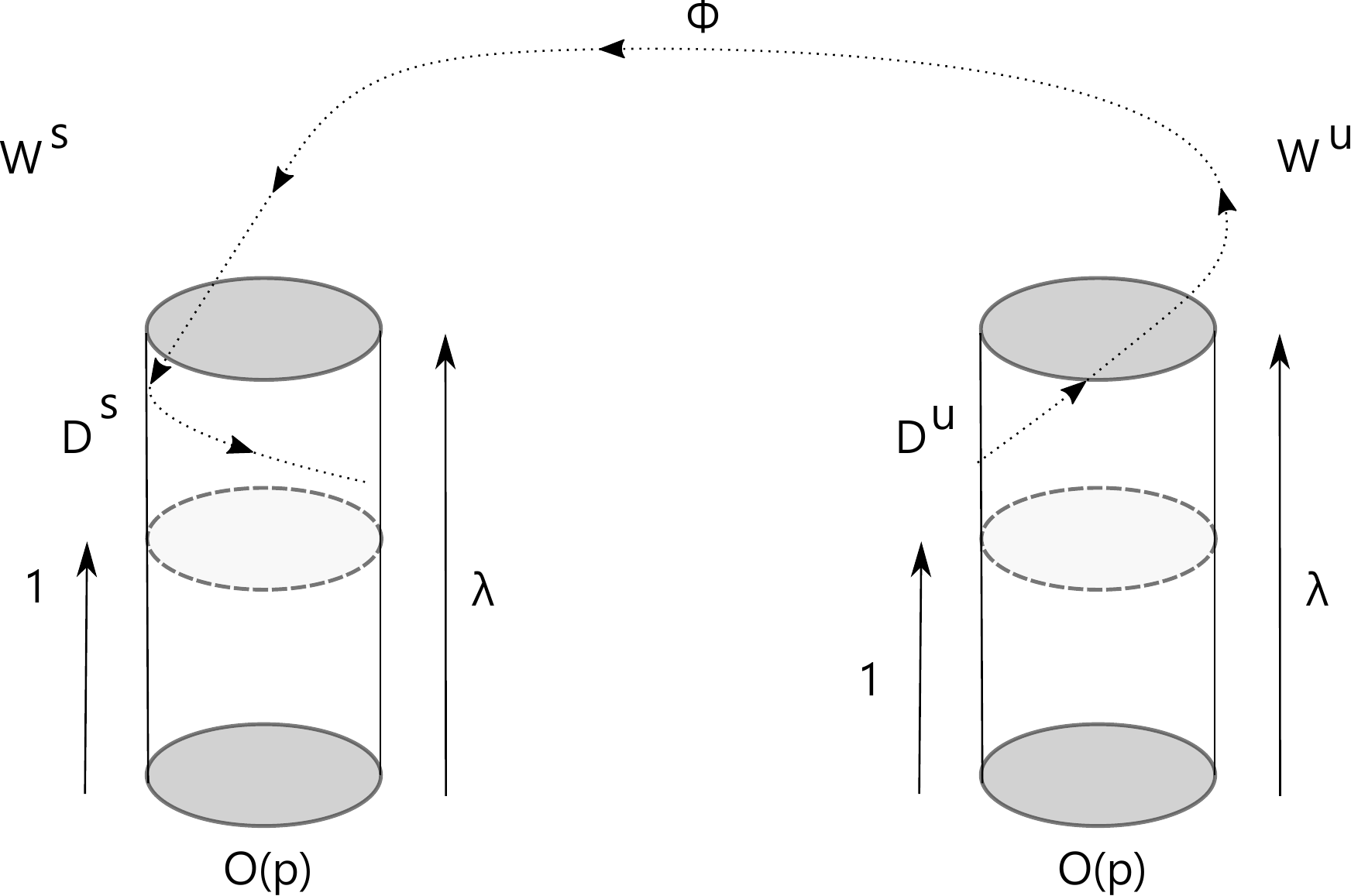}
	\caption{Construction of $D^s,D^u$. Here $\lambda=\lambda_u=\lambda_s^{-1}$.}
	\label{fig:drawing1}
\end{figure}

In the $(x,t)$ coordinates the flow $\phi^c_t$ is the solution to the differential equation $\dot{\theta}=1,\dot{x}=\alpha x$, and similarly for the $(\theta,y)$ coordinates. We deduce that $\phi^c_t$ is given by
\begin{align}
\begin{cases}
\theta\mapsto \theta+t\\
x\mapsto xe^{\alpha t}
\end{cases}\quad \text{in } W^u_{loc}(O(p), \phi)\\
\begin{cases}
\theta\mapsto \theta+t\\
y\mapsto ye^{\beta t}
\end{cases}\quad \text{in } W^s_{loc}(O(p), \phi)	
\end{align}
On the other hand, the diffeomorphism $f$ acts in the vertical coordinates simply by multiplying by $\la_u,\la_s$, 
\begin{align}
f(\theta,x)=(\theta',\lambda_ux)\\
f(\theta,y)=(\theta',\lambda_sy)
\end{align}

We now consider the  homoclinic trajectories of $\phi^c_t$ connecting $f(\gamma_u)$ with $\gamma_s$.

\begin{lemma}
Any such homoclinic trajectory is fixed by $f$.	
\end{lemma}	

\begin{proof}
For a homoclinic trajectory $O(q)$ as before we denote $X(q)\in f(\gamma_u)\cap O(q)$, $L(q)=$ smallest time such that $Y(q)=\phi_{L(q)}(X(q))\in \gamma_s$ and we observe that a given $M>0$ the number of homoclinic trajectories $O(q)$ with $L(q)\leq M$ is finite, hence, as $f$ is isometry in the flow direction, it suffices to show that the possibles $L(q)$ are bounded. 

Take an homoclinic curve $O(q_0)$ of minimal length and denote $x_0,y_0$ the second coordinates of $X(q_0),Y(q_0)$. Let $k_0\in \mathbb{Z}$ be the smallest integer such that $f^k_0(X(q_0))\in D^s$ and define $Y_1=f^k_0(X(q_0))$, and $X_1$ the point in $f(\gamma_u)\cap O(Y_1)$ of minimal length, which we denote by $L_1$ (i.e.\@ $\phi_{L_1}(X_1)=Y_1$). Similarly, $x_1,y_1$ denote the second coordinates of $X_1,Y_1$ respectively.

The oriented orbit segment joining $Y_1=f^{k_0}(X_0)$ with $f^{k_0}(X_0)$ is completely contained in $W^s_{loc}(O(p),\phi^c_t)$ and has length $L_0$ (because $f$ is an isometry in the flow direction), thus we deduce
\[
\lambda^{-k_0}y_0=y_1e^{\beta L_0}
\]
On the other hand and arguing analogously, the oriented orbit segment joining $f^{-k_0}(X_1)$ with $X_0$ is completely contained in $W^u_{loc}(O(p),\phi^c_t)$ and has length $L_1$, hence
\[
x_0e^{\beta L_1}=\lambda^{-k_0}x_1,
\] 
thus combining the two previous equations we deduce
\[
\frac{y_1}{y_0}e^{\beta L_0}=\frac{x_0}{x_1}e^{\beta L_1}\Rightarrow e^{\beta(L_1-L_0)}=\frac{x_1y_1}{x_0y_0}
\]
We now argue inductively (with the natural definition for $x_j,y_j$) and obtain 
\begin{align*}
&L_j-L_{j-1}=\alpha(\ln x_jy_j-\ln x_{j-1}y_{j-1} )\\
&\Rightarrow L_j-L_0=\alpha(\ln x_jy_j-\ln x_{0}y_{0}) 	\quad\forall j\geq 1.
\end{align*}
Noting that $x_jy_j\in [1,\lambda^2]$ for every $j$ we obtain that $L_j$ is bounded in $j$, as claimed.
\end{proof}

We are ready to finish the proof.

\begin{proof}[Proof of Proposition \ref{flow}]
 It follows that $f$ fixes an orbit $O(q)\neq O(p)$ homoclinic to $O(p)$. It follows that there is  $T$ positive such that $\phi_T^c(X_0)=f^{k_0}(X_0)$, hence $\lambda_s^k.y_0= \exp(\beta.T)\cdot y_0$ which implies $\lambda_s= \exp(\beta. \frac{T}{k})$, and using that $\lambda_u=\lambda^{-1}_s$ we get $\lambda_u=\exp(\al.T)$. Finally, using the linearizing coordinates we deduce that $Df=D\phi_{\frac{T}{k}}^c$, and since $f$ fixes two orbits in these coordinates, $f=\phi_{\frac{T}{k}}^c$ in $W^u(O(p),\phi^c_t)\cup W^s(O(p),\phi^c_t)$, which implies, since the stable and unstable manifolds of $O(p)$ are dense, that $f=\phi_{\frac{T}{k}}^c$ on $M$.
\end{proof}

\section*{Acknowledgments}

The authors would like to thank the referee for all the valuable input and the corrections that improved the manuscript.
\bibliographystyle{alpha}
\bibliography{bibliografia}

\end{document}